\documentclass[12pt]{amsart}
\usepackage{amsmath, amsthm, amscd, amsfonts, amssymb, latexsym, graphicx, color}

\usepackage[utf8]{inputenc}

\usepackage{amsmath,amsthm}
\usepackage{amssymb,latexsym}
\usepackage{enumerate}

\newtheorem{theorem}{Theorem}[section]
\newtheorem{lemma}[theorem]{Lemma}
\newtheorem{proposition}[theorem]{Proposition}
\newtheorem{corollary}[theorem]{Corollary}
\newtheorem{observation}[theorem]{Observation}
\theoremstyle{definition}

\newtheorem{example}[theorem]{Example}

\newtheorem{claim}[theorem]{Claim}
\theoremstyle{remark}
\newtheorem{remark}[theorem]{Remark}
\numberwithin{equation}{section}

\begin{document}

\title[Extension of definable Lipschitz maps]{Extension of Lipschitz maps  \\
       definable in Hensel minimal structures}

\author[Krzysztof Jan Nowak]{Krzysztof Jan Nowak}




\date{}

\begin{abstract}
In this paper, we establish a theorem on extension of Lipschitz maps $f$ definable in Hensel minimal fields $K$. This may be regarded as a definable, non-Archimedean, non-locally compact version of Kirszbraun's extension theorem. We shall proceed with double induction with respect to the dimensions of the ambient space and of the domain of $f$. To this end, we introduce the concept of a definable open cell package with a skeleton which, along with the concept of a risometry, plays a key role in our induction procedure.
\end{abstract}

\maketitle


\section{Introduction}


We are concerned with a 1-h-minimal, non-trivially valued field $K$ of equicharacteristic zero. In other words, $K$ is a model of a 1-h-minimal (complete) theory $T$ in an expansion $\mathcal{L}$ of the pure valued field language $(K,0,1,+,-,\cdot,\mathcal{O}_{K})$.
The axiomatically based theory of Hensel minimal structures was recently introduced by Cluckers--Halupczok--Rideau~\cite{C-H-R}.

\vspace{1ex}

Throughout the paper, we shall adopt multiplicative convention for valuation $| \cdot |: K^{\times} \to |K^{\times}|$, where $K^{\times} := K \setminus \{ 0 \}$.

\vspace{1ex}

A Lipschitz continuous map $f: A \to K^m$, $A \subset K^n$, with Lipschitz constant $\epsilon \in |K|$, $\epsilon \neq 0$, shall be called an $\epsilon$-Lipschitz map.

\vspace{1ex}

The main aim of this paper is to establish a theorem on extension of Lipschitz definable maps from a subset $A$ of $K^{n}$ to an affine space $K^{m}$ (Theorem~\ref{ext}), which may be regarded as a definable, non-Archimedean, non-locally compact version of Kirszbraun's extension theorem~\cite{Kir}.

\vspace{1ex}

The extention of Lipschitz continuous functions $f: A \to \mathbb{R}$ from a subset $A$ of $\mathbb{R}^{n}$ with the same constant, say $1$, goes back to McShane and
Whitney. The more difficult and delicate case of maps with values in $\mathbb{R}^{k}$ was achieved by Kirszbraun~\cite{Kir}.

\vspace{1ex}

In non-Archimedean geometry, the following version of Kirsz\-braun's theorem was established (using Zorn's lemma) by Bhaskaran~\cite{Bas}:  

\vspace{1ex}

\begin{em}
Let $K$ be a rank one valued field, $X$ be an ultrametric space, and $A \subset X$. Then every bounded Lipschitz function $f:A \to K$ extends to a bounded Lipschitz function $F:X \to K$, with the Lipschitz constant and supremum norm preserved, whenever the field $K$ or the subspace $A$ is spherically complete.
\end{em}

\vspace{1ex}

Moreover, he proved that the field $K$ is spherically complete if and only if the above extension theorem holds for every ultrametric space $X$ and every $A \subset X$. Therefore, it is not always possible to extend Lipschitz functions to the whole ultrametric space with the same Lipschitz constant.
Afterwards, the paper~\cite[Theorem~2.9]{BDHM} proves the existence of Lipschitz retractions, with an arbitrary Lipschitz constant $>1$, from an ultrametric space $X$ on an arbitrary closed subset $A$ of $X$. Hence follows Lipshitz extension with an arbitrarily small magnification of the Lipschitz constant.

\vspace{1ex}

Consider a 1-Lipschitz retraction $r:K^{n} \to A$ on a closed subset $A$ of $K^{n}$. It is easy to check that then the distance from $A$ is realized:
\begin{equation}\label{dist}
 |x-r(x)| = \inf \, \{ |x-a|: \ a \in A \}, \ \  x \in K^{n}.
\end{equation}
Whenever the valuation is non-discrete of rank one, there exist closed (and definable too) subsets $A \subset K$ which are not 1-Lipschitz retracts of the line $K$. Furthermore, if the valuation is of rank greater than 1, there exist closed (and definable too) subsets $A \subset K$ which are not Lipschitz retracts of $K$ with an arbitrary Lipschitz constant $>1$.
This is demonstrated by the following

\begin{example}\label{ex-ret}
Suppose that the multiplicative group $|K^{\times}|$ is densely ordered of rank one, or is of rank greater than 1. In the former case (for instance, for $K= \mathbb{R}((t))$ or $K= \mathbb{C}((t))$), let $\rho \in |K^{\times}|$, $\rho >0$. In the latter case, let $G$ be a proper (definable) convex subgroup of $|K^{\times}|$. Define the subsets
$$ A := \{ x \in K: \ |x|> \rho \} \ \ \text{and} \ \  B:= \{ x \in K: \ |x| > G \}, $$
depending on the case in question. By equality~\ref{dist}, the set $A$ is not a 1-Lipschitz retract of $K$, and $B$ is not a Lipschitz retract of $K$ with any Lipschitz constant $\epsilon >1$ from $G$.
\end{example}

\vspace{1ex}

Aschenbrenner--Fischer~\cite{Asch} achieved a definable, real version of Kirsz\-braun's theorem for definably complete expansions of ordered
fields.

\vspace{1ex}

To our best knowledge, the only definable, non-Archimedian version of Kirszbraun's theorem was achieved by Cluckers--Martin~\cite{C-M} in the $p$-adic, thus locally compact case. And more precisely, for functions which are semi-algebraic, subanalytic or definable in an analytic structure on a finite extension of the field $\mathbb{Q}_{p}$ of $p$-adic numbers, for whose theory we refer to the papers~\cite{De-Dries,C-Lip-0}). They prove this theorem and the
existence of a definable 1-Lipschitz retraction on any closed definable subset $A$ of
$\mathbb{Q}_{p}^{n}$, proceeding with simultaneous induction on
the dimension $n$ of the ambient space. Their construction of definable retractions makes use
of some definable Skolem functions.

\vspace{1ex}

Also, they posed the question (see ibid.\ Remark~3) whether their $p$-adic version of Kirsz\-braun's theorem and the existence of Lipschitz retractions on $p$-adic closed definable subsets hold in some form for other classes of valued fields; with natural examples $\mathbb{R}((t))$ and $\mathbb{C}((t))$. And they indicated that some difficulties in more general settings are the absence of definable Skolem functions in general and infiniteness of the residue field. The easier case of Lipschitz extension of definable $p$-adic functions on the affine line $\mathbb{Q}_{p}$ was treated in \cite{Kui}.

\vspace{1ex}

The main aim of this paper is to establish the theorem on extending definable 1-Lipschitz maps stated below, which distinguishes two cases depending on whether the subset of elements $>1$ in the value group $vK$ has the minimal element $\epsilon$ or not. For this purpose, we introduce the concept of a definable open cell package with a skeleton, which in a sense compensate for the lack of definable algebraic Skolem functions.

\begin{theorem}\label{ext}
Let $f: A \to K^m$ be a $0$-definable 1-Lipschitz map on a (possibly non-closed) subset $A \subset K^n$ of dimension $k$.

I. Suppose the value group $|K|$ is densely ordered, i.e.\ has no minimal element among the elements $>1$.
Then, $f$ extends to a $0$-definable $\epsilon$-Lipschitz map $F: K^n \to K^{m}$ for any $\epsilon \in |K|$, $\epsilon > 1$.


II.  Suppose the value group $|K|$ has the minimal element $\epsilon$ among the elements $>1$. Then $f$ extends to a $0$-definable $\epsilon^{\omega(k)}$-Lipschitz map $F: K^n \to K^{m}$, where $\omega(k) = 2^{k-2}$ if $k\geq 2$, and $\omega(1)=\omega(0)=0$.
\end{theorem}

Since the affine spaces $K^{m}$ are equipped with the maximum norm, it suffices to prove the theorem for the case $m=1$. 

\vspace{1ex}

The Lipschitz constant has been increased because the initial 1-Lipschitz function $f$ is replaced by a function $g$ with the condition:
\begin{equation}\label{risometry}
rv ( g(x_1, \ldots, x_{k-1},x_{k} + y_{k}, x_{k+1},\ldots,x_n) - g(x)) = rv (y_{k})
\end{equation}
for a single, distinguished variable $x_{k}$; most often it is the last variable of the package under consideration. In other words, $g$ is a risometry onto its image with respect to the variable $x_{k}$, the concept introduced by Halupczok~\cite{Hal}.

\begin{remark}\label{rem-risometry}
The key property of a risometry we use is that the image of an open ball is an open ball of the same radius. This immediately follows from~\cite[Lemma~2.33]{Hal} and the fact 0-h-minimality implies definable spherical completeness (cf.~\cite[Lemma~2.1.7]{C-H-R}).
\end{remark}


To get such a function $g$, take an element $\varepsilon \in K$ such that $|\varepsilon| = \epsilon$ and put
$$ g: A \to K, \ \ \ g(x) := \, x_{k} + 1/\varepsilon \cdot f(x). $$
Clearly, once we find a 0-definable 1-Lipschitz extension $G: K^n \to K$ of $g$, the function
$$ F(x) := \varepsilon ( G(x) - x_{k} ) $$
is an extension we are looking for. So from now on it will be assumed that a given function $f(x)$ satisfies a condition of type~\ref{risometry}.

\vspace{1ex}

Actually, for the sake of the proof of the extension theorem, we need the following stronger version, which is uniform with respect to parameters from the sort $RV$. We state it only for the case~I.

\begin{theorem}\label{ext-uni}
Consider a 0-definable family
$$ f_{\lambda}: A_{\lambda} \to K, \ \ \lambda \in \Lambda \subset (RV)^s, $$
of 1-Lipschitz functions on subsets of $K^n$. Then, for any $\epsilon \in |K|$, $\epsilon > 1$, there exists a 0-definable family
$$ F_{\lambda}: K^{n} \to K, \ \ \lambda \in \Lambda $$
of $\epsilon$-Lipschitz functions that extend the functions $f_{\lambda}$ for $\lambda \in \Lambda$.
\end{theorem}

We shall prove the above extension theorem by double induction on the dimension $n$ of the ambient space $K^n$ and on the dimension $k$ of the subset $A \subset K^n$. Each induction step of type~$(n,k)$, with $n \geq 1$ and $1 \leq k \leq n$, requires the induction hypothesis of type~$(n,k-1)$ (ordinary version) and of type~$(k-1,k-1)$ (uniform version).

\vspace{1ex}

While the ordinary version of the extension theorem will be applied to a given 1-Lipschitz function $f$ restricted to some 0-definable subset in the affine space under consideration, the uniform one to some 0-definable family of 1-Lipschitz centers in order to extend them to global definable 1-Lipschitz ones. This, in turn, is used to obtain, via partitioning, open cell packages which satisfy a certain necessary package property.


\vspace{1ex}


Notice that each induction step, both of the first and second kind, increases the Lipschitz constant by any factor $\epsilon > 1$ if the value group $vK$ has no minimal element among the elements $>1$, and by some power $\epsilon^{\omega}$ if $\epsilon$ is the minimal element from among the elements $>1$. This is obvious for the first kind.

\vspace{1ex}

The second kind requires the following observation. Suppose that, after extension, we have got a new global $\epsilon$-Lipschitz center $c(x')$, $x' = (x_{1},\ldots, x_{k})$. Applying the homothetic substitution of $x/\varepsilon$ for $x$ with $|\varepsilon| = \epsilon$, we get a new, this time 1-Lipschitz center, and the new 1-Lipschitz function $f_{1}(x)$ given by the formula
$$ f_{1}(x) := f(x'/\varepsilon,x'') \ \ \text{where} \ \ x'' = (x_{k+1},\ldots, x_{n}). $$
Then, given a global 1-Lipschitz extension $F_{1}(x)$ of $f_{1}(x)$, the function
$$ F(x) := F_{1}(\varepsilon x',x'')) $$
is the desired global extension of $f$ which is merely $\epsilon$-Lipschitz.

\vspace{1ex}

The exponent $\omega = \omega(n,k)$ depends on the type $(n,k)$ of induction step. Clearly, we get the following formulae
$$ \omega(n,0)=0, \ \ \omega(n,k) = \omega(n,k-1) + \omega(k-1,k-1) \ \ \text{if} \ 1 \leq k \leq n, $$
and hence $\omega(n,k) = \omega(k) = 2^{k-2}$ if $k \geq 2$, and $\omega(n,0) = \omega(n,1) =0$.

\vspace{1ex}

In the proof, the increase of the Lipschitz constant will be clearly indicated during the induction steps of the first type~$(n,k-1)$ (ordinary version). For simplicity, we shall omit this in the induction step of the second type (uniform version). But all the induction steps are taken into account when calculating the Lipschitz constant.  




\vspace{1ex}

\section{Cell decomposition in Hensel minimal structures}

We begin by introducing some necessary notation and terminology.
Denote by $\mathcal{O}_{K}$ and $\mathcal{M}_{K}$ the valuation ring of the valued field $K$ and its maximal ideal, respectively.
We use multiplicative convention for the value group and valuation map:
$$ | \cdot |: K^{\times} \to |K|, \ \ |0| := 0. $$
The ultra-metric norm on the affine space $K^n$ is the maximum norm
$$ |(x_{1},\ldots,x_{n})| := \max \, \{ |x_{1}|, \ldots, |x_{n}| \}. $$
The auxiliary sort:
$$ RV(K) := G(K) \cup \{ 0 \}, \ \ G(K) := K^{\times}/(1+\mathcal{M}_{K}), $$
plays an essential role in geometry and model theory of valued fields. In particular, it provides parameters for definable families.

\vspace{1ex}

We adopt the following convention: the words 0-definable and $X$-definable mean $\mathcal{L}$-definable and $\mathcal{L}_{X}$-definable, i.e.\ definable with parameters from $X$.

\vspace{1ex}

For $m \leq n$, denote by $\pi_{\leq m}$ or $\pi_{< m+1}$ the projection $K^{n} \to K^{m}$ onto the first $m$ coordinates; put $x_{\leq m} = \pi_{\leq m}(x)$. Let $C \subset K^{n}$ be a non-empty 0-definable set, $j_{i} \in \{ 0, 1 \}$ and
$$ c_{i} : \pi_{<i}(C) \to K $$
be 0-definable functions $i=1,\ldots,n$. Then $C$ is called a 0-definable cell with center tuple $c = (c_{i})_{i=1}^{n}$ and of cell-type $j =(j_{i})_{i=1}^{n}$ if it is of the form:
\begin{equation}\label{cell}
  C = \left\{ x \in K^{n}: (rv(x_{i} - c_{i}(x_{<i})))_{i=1}^{n} \in R \right\},
\end{equation}
for a (necessarily 0-definable) set
$$ R \subset \prod_{i=1}^{n} \, j_{i} \cdot G(K), $$
where $0 \cdot G(K) = {0} \subset RV(K)$ and $1 \cdot G(K) = G(K) \subset RV(K)$. One can similarly define $A$-definable cells.

\vspace{1ex}

In the absence of algebraic Skolem functions or, equivalently, of the condition that algebraic closure and definable closure coincide in $T$, 
i.e.\ the algebraic closure $\mathrm{acl}\,(A)$ equals the definable closure $\mathrm{dcl}\,(A)$ for any Henselian field $K' \equiv K$ and every $A \subset K'$, the following concept of reparameterized cells must come into play.

\vspace{1ex}

Consider a 0-definable function $\sigma: C \to RV(K)^{s}$. Then $(C,\sigma)$ is called a 0-definable reparameterized (by $\sigma$) cell if each set $\sigma^{-1}(\xi)$,  $\xi \in \sigma(C)$, is a $\xi$-definable cell with some center tuple $c_{\xi}$ depending definably on $\xi$ and of cell-type independent of $\xi$. One can always modify $\sigma$ so that each $\sigma^{-1}(\xi)$ is either empty or a single twisted box.

\vspace{1ex}

We have the following fundamental theorem on Lipschitz cell decomposition compatible with $RV$-parameters (cf.~\cite[Theorem~5.7.3]{C-H-R}).

\begin{theorem}\label{cell-decomposition}
For every 0-definable sets
$$ X \subset K^{n} \ \ \text{and} \ \ P \subset X \times RV(K)^{t}, $$
there exists a finite decomposition of $X$ into 0-definable reparametrized cells $(C_{k},\sigma_{k})$ such that the fibers of $P$ over each twisted box of each $C_{k}$ are constant or, equivalently, the fiber of $P$ over each $\xi \in RV(K)^{t}$ is a union of some twisted boxes from the cells $C_{k}$.

Furthermore, one can require that each $C_{k}$ is, after some coordinate permutation, a reparametrized cell of type $(1,\ldots,1,0,\ldots,0)$ with 1-Lipschitz centers
$$ c_{\xi} = (c_{\xi,1},\ldots,c_{\xi,n}), \ \ \xi \in \sigma_{k}(C_{k}). $$
Such cells $C_{k}$ shall be called 0-definable reparametrized Lipschitz cells.    \hspace*{\fill} $\Box$
\end{theorem}

Under the assumptions of the above theorem, we may regard $P$ as a 0-definable family of subsets of $X$ parametrized by $RV(K)^t$, and say that the cell decomposition into cells $C_{k}$ is compatible with that family. This also ensures the existence of cell decompositions compatible with imaginary parameter from the auxiliary sort $RV$. We shall use it in Sections~5 and~6 for construction of a skeleton of a 0-definable open Lipschitz cell in higher dimensions.

\vspace{1ex}

We still need the following

\begin{proposition}\label{graph}
Every 0-definable subset $A \subset K^n$ is a finite disjoint union of 0-definable sets $E$ that are, after some permutation of the variables, disjoint unions of the form
$$ E := \bigcup_{\eta \in \tau(D)} \, \bigcup_{j=1}^{d} \, \mathrm{graph}\, \phi_{\eta,j} \subset K^n, $$
where
$$ D = \bigcup_{\eta \in \tau(D)} \, D_{\eta} \subset K^{k}_{x_{\leq k}} $$
is a 0-definable reparametrized Lipschitz open cell in $K^{k}_{x_{\leq k}}$,
$$ \phi_{\eta,j}: D_{\eta} \to K^{n-k}_{x_{> k}}, \ \ \eta \in \tau(D), \ j =1,\ldots,d, $$
are 1-Lipschitz functions, and the set
$$ \bigcup_{j=1}^{d} \, \mathrm{graph}\, \phi_{\eta,j} \subset K^n $$
is $\eta$-definable for each $\eta \in \tau(D)$.
\end{proposition}

\begin{proof}
We shall apply cell decomposition 
with reparametrized cells parametrizing single twisted boxes. We may assume that $A$ is, after some coordinate permutation, a 0-definable reparametrized Lipschitz cell $C$ of dimension $k$ of type $(1,\ldots,1,0,\ldots,0)$. Put
$$ D := \pi_{\leq k}(C) = \bigcup_{\xi} \, D_{\xi} \ \ \text{with} \ \ D_{\xi} := \pi_{\leq k}(C_{\xi}) \subset K_{x_{\leq k}}. $$
The restriction
$$ p: \bigcup_{\xi} \, \mathrm{graph} (c_{\xi,k+1}, \ldots, c_{\xi,n}) \to D $$
of the projection $\pi_{\leq k}$  has of course finite, thus uniformly bounded fibres; say, of maximum cardinality $l$.
Since $D$ is the union of the interiors $\mathrm{int} (D_{d})$ of the sets
$$ D_{d} := \{ x_{1} \in D: \; \# \, p^{-1}(x_{1})= d \}, \ \ d \leq l, $$
and of a set of dimension $< k$, we can assume, by a routine induction argument, that $D = \mathrm{int} (D_{d})$, and thus that the constant cardinality of the fibers of $p$ is $d$. Therefore the 0-definable family
$$ D_{\bar{\xi}} := \left\{ u \in \bigcap_{j=1}^{d} \, D_{\xi_{j}}: \ \# \, \{ c_{\xi_1,>k}(u),\ldots,c_{\xi_d,> k}(u) \} = d \right\}, $$
with $\bar{\xi} = (\xi_1,\ldots,\xi_d) \in \sigma(C)^{d}$, covers $D$, and even is its partition. Indeed, the first assertion follows from that we have
$$ D_{\lambda} = \bigcup_{\lambda \in \bar{\xi}} \, D_{\bar{\xi}}. $$
for $\lambda \in \sigma(C)$, and the second from that a reprametrized cell is a disjoint union of cells.

\vspace{1ex}

Now it follows from Theorem~\ref{cell-decomposition} that $D$ is a finite disjoint union of 0-definable reparametrized Lipschitz cells $D_{l}$ compatible with that family.
Throwing away pieces of lower dimension and traeting them by induction, we can thus assume that
$$ D = \bigcup \, \left\{  D_{\eta}: \ \eta \in  \tau (D)   \right\},  $$
is one of the open reparametrized cells of the decomposition. Then, since each cell (actually, a single twisted box) $D_{\eta}$ is contained in exactly one of the sets $D_{\bar{\xi}}$ with $\bar{\xi} = (\xi_1,\ldots,\xi_d)$, we can put
$$ \phi_{\eta,j} := c_{\xi_j}|D_{\eta}: D_{\eta} \to K^{n-k}, \ \ \eta \in \tau(D), \ j=1,\ldots,d. $$
Then the set
$$ E := \bigcup_{\eta \in \tau(D)} \, \bigcup_{j=1}^{d} \, \mathrm{graph}\, \phi_{\eta,j} \subset K^n, $$
is $\eta$-definable and we are done.
\end{proof}

\begin{remark}\label{rem-graph}
In the notation of the above proof, each function $\phi_{\eta,j}$ with $\eta \in \tau(D)$ and $d=1,\ldots,d$ is $(\eta, \bar{\xi})$-definable if $D_{\eta} \subset D_{\bar{\xi}}$. Moreover, the set $\bar{\xi}$ is $\eta$-definable.
\end{remark}

\vspace{1ex}

In the easier case where definable algebraic Skolem functions are available, in analogy with o-minimal geometry, cell decomposition is much more of finitary character (cf.~\cite{C-H-R}, Theorem~5.2.4 with the following addenda). Then we can assume that the subset $A$ is an ordinary 0-definable cell. In particular, the case where the subset $A$ is of dimension 0 comes down to the case of a single 0-definable point, and the above proposition can be formulated in terms of finite unions of the graphs of 0-definable 1-Lipschitz functions.

\vspace{1ex}

The general case, however, is much more involved. A key role in the Lipschitz extension problem is then played by a concept of a 0-definable package with a skeleton induced by an open reparametrized cell. We shall introduce this concept recursively to satisfy a metric condition relating the distance between the centers and the size of the twisted boxes: first on the affine line $K$, and next in higher dimensions, which is subtler yet. In the absence of algebraic Skolem functions, when reparametrized cells must come into play, local Lipschitz continuity does not imply piecewise Lipschitz continuity in general. These two features are reflected in the construction of a package itself.

\vspace{1ex}


\section{Preliminary results}

Now we give some results about the possibility of finite definable partitioning when extending Lipschitz functions, which are necessary throughout the paper.



\begin{proposition}\label{part}
Let $A$ and $B$ be arbitrary 0-definable subsets of $K^n$ and $f: A \cup B \to K$ a 0-definable 1-Lipschitz function which vanishes on $B$. If the restriction $f|A$ of $f$ to $A$ extends to a 0-definable 1-Lipschitz function $F: K^n \to K$, so does $f$ to a 0-definable 1-Lipschitz function $G: K^n \to K$.
\end{proposition}

\begin{proof}
We begin by formulating two symmetric conditions for the points of $x \in K^n$:

\vspace{1ex}

(1) \ \  $\forall \, b \in B \; \exists \, a \in A \ \ |x-a| < |x-b|$;

(2) \ \  $\forall \, a \in A \; \exists \, b \in B \ \ |x-b| < |x-a|$.

\vspace{1ex}

Note that the condition
$$ (3) := \neg [ (1) \vee (2) ] $$
is equivalent to
$$ \exists \, b_0 \in B, a_0 \in A  \; \forall \, a \in A, b \in B \ \ |x-a| \geq |x-b_0| \ \wedge \ |x-b| \geq |x-a_0|, $$
and thus condition (3) means that
$$ \exists \, b_0 \in B, a_0 \in A  \; \forall \, a \in A, b \in B \ \ |x-a|, |x-b| \geq |x-a_0| = |x-b_0|. $$

It is easy to check that condition~(1) holds on $A \setminus B$, condition~(2) on $B \setminus A$, and condition~(3) on $A \cap B$. Hence
the function
$$ G(x) = \left\{ \begin{array}{cl}
                        F(x) & \mbox{ if } \ x \ \mbox{ satisfies } [(1) \wedge \neg (2)], \\
                        0 & \mbox{ if } \ x \ \mbox{ satisfies } [(2) \vee (3)].
                        \end{array}
               \right.
$$
is an extension of $f$.
We shall show that $G$ is a 1-Lipschitz function. It suffices to analyze the value $|G(x) -G(y)|$ in the only two non-trivial cases where:

\vspace{1ex}

I. $x$ satisfies $[(1) \wedge \neg (2)]$ and $y$ satisfies $(2)$,

\noindent
or

II. $x$ satisfies $[(1) \wedge \neg (2)]$ and $y$ satisfies $(3)$.

\vspace{2ex}

{\em Case I.} We first show that
$$ \exists \, b \in B \ \ |y-b| \leq |x-y|. $$
Otherwise
$$ \forall \, b \in B \ \ |y-b| > |x-y|. $$
By condition~(2), for any $a \in A$ there is an element $b \in B$ such that $|y-b| < |y-a|$. Then
$$ |y-b| = |x-b|, \ \ |y-a| = |x-a| \ \text{ and } \ |x-b| < |x-a|. $$
Therefore $x$ satisfies condition (2), which is a contradiction.

By condition~(1), we get $|x-a| < |x-b|$ for some $a \in A$. Then
$$ |F(x) - f(a)| \leq |x-a| < |x-b|, $$
$$ |f(a)| = |f(a) - f(b)| \leq |a-b| = |x-b| \ \ \text{and} \ \ |F(x)| \leq |x-b|.  $$
Hence
$$ |G(x) - G(y)| = |F(x) - 0| =|F(x)| \leq |x-b| \leq |x-y|, $$
as desired.

\vspace{1ex}

{\em Case II.} Since $y$ satisfies condition (3), it follows that
$$ \exists \, b_0 \in B, a_0 \in A  \; \forall \, a \in A, b \in B \ \ |y-a|, |y-b| \geq |y-a_0| = |y-b_0|. $$

Were
$$ |x-y| < |y-a_0| = |y-b_0|, $$
we would get for every $a \in A$ the inequalities
$$ |x-y| < |y - a_0| \leq |y-a| \ \ \text{and} \ \ |x-y| < |y-b_0| = |x-b_0|. $$
Hence
$$ |x-a| = |y-a| \geq |y - b_0| = |x - b_0|, $$
and thus $x$ does not satisfy condition (1), which is a contradiction.

Consequently, we get
$$ |x-y| \geq |y-a_0| = |y-b_0|, $$
and thus
$$ |x-a_0| \leq |x-y|, \ \ |x-b_0| \leq |x-y| \ \  \text{and} \ \ |a_0 - b_0| \leq |x-y|. $$
Hence
$$ |F(x) - f(a_0)| \leq |x-a_0|, \ \ |f(a_0)| = |f(a_0) - f(b_0)| \leq |a_0 - b_0|, $$
and thus
$$ |F(x)| \leq \max \, \{ |x-a_0|, |a_0 - b_0| \} \leq |x-y|.   $$
Therefore
$$ |G(x) - G(y)| = |F(x)| \leq |x-y|, $$
as desired. This finishes the proof.
\end{proof}

\vspace{1ex}

Observe more precisely that the points $x$ from
$$ \overline{A} \setminus \overline{B}, \ \ \overline{B} \setminus \overline{A}, \ \ (\overline{A} \setminus A) \cap (\overline{B} \setminus B), \ \
   (\overline{A} \setminus A) \cap B, \ \ A \cap (\overline{B} \setminus B), \ \ A \cap B, $$
satisfy respectively the following conditions
$$ (1) \wedge \neg (2), \ \ (2) \wedge \neg (1), \ \ (1) \wedge (2), \ \ (2) \wedge \neg (1), \ \ (1) \wedge \neg (2), \ \ (3); $$
here $\overline{A}$ denotes the topological closure of a set $A$. Therefore the value $G(x)$ at the points $x$ from the above subsets of $K^n$ is respectively equal to
$$ F(x), \ \ 0, \ \ 0, \ \ 0, \ \ F(x), \ \ 0. $$
Hence
$$ G|(A \setminus \overline{B}) = F|(A \setminus \overline{B}), \ \ G|(B \setminus \overline{A}) = 0, \ \ G|((\overline{A} \setminus A) \cap  B) =0, \ \ $$
$$ G|(A \cap (\overline{B} \setminus B)) = F|(A \cap (\overline{B} \setminus B)), \ \ G|(A \cap B) = 0. $$

\vspace{1ex}

We say that a 0-definable subset $A$ of $K^n$ has the {\em Lipschitz extension property}, abbreviated to {\em LE-property}, if every 0-definable 1-Lipschitz function $f:A \to K$ extends to a 0-definable 1-Lipschitz function  $F: K^n \to K$.

\begin{corollary}\label{LEP} (Partition Lemma)
Consider a finite number of 0-definable subsets $A_1,\ldots,A_s$ of $K^n$. If each of them has the LE-property, so does their union $A_1 \cup \dots \cup A_s$.
\end{corollary}

\begin{proof}
Induction with respect to the number $s$ reduces the proof to the case $s=2$. Then the conclusion follows directly from Proposition~\ref{part}. Indeed, let $f: A_1 \cup A_2 \to K$ be a 0-definable 1-Lipschitz function, $F_2$ be a 0-definable 1-Lipschitz extension of the restriction $f|A_2$. Then the function
$$ g := f - F_2|(A_1 \cup A_2) $$
vanishes on $A_2$, and thus, by Proposition~\ref{part}, $g$ extends to a 0-definable 1-Lipschitz function $F_1: K^n \to K$. Then the function $F := F_1 + F_2$ is an extension of the initial function $f$, concluding the proof.
\end{proof}

\vspace{1ex}

\section{Lipschitz extension on the affine line}

If $A \subset K$ is of dimension zero, then $A$ is a finite subset of $K$, and it is easy to check that the function
\begin{equation}\label{ext-fin-0}
  F: K \ni x \mapsto \frac{\sum f(A(x))}{\# \, A(x)} \in K,
\end{equation}
where
$$ A(x) := \{ a \in A: \ |x-a| = \min \, \{ |x-s|: \ s \in A \} \}, $$
is a 0-definable 1-Lipschitz extension of $f$ we are looking for.

\vspace{1ex}

By Corollary~\ref{LEP}, it suffices to consider the case where
$$ A=C = \bigcup_{\xi} \, C_{\xi} $$
is an 0-definable reparametrized open cell; and further, as explained in Section~1, that
$f: C \to K$ satisfies condition~\ref{risometry}, which means that $f$ is a risometry onto its image.

\vspace{1ex}

The concept of a 0-definable reparametrized open cell on the affine line is quite simple in comparison with higher dimensions. Indeed, the set of centers $c_{\xi} \in K$ of $C$ is finite, and consists of some $d$ distinct points $\{ c_1,\ldots,c_d \} \subset K$. Put
$$ C_{i} := \bigcup \, \{ C_{\xi}: \ c_{\xi} = c_{i} \} \subset K $$
and
$$ R_{i} := \bigcup \, \{ R_{\xi}: \ c_{\xi} = c_{i} \} \subset G(K) \subset RV(K); $$
clearly each set $C_{i}$ and $R_{i}$ is $c_{i}$-definable for $ i=1,\ldots,d$. We thus get the following

\begin{observation}\label{cell-fin}
Every 0-definable reparametrized open cell $C$ in $K$ is the disjoint union of a finite 0-definable family of open cells $C_{i}$, $i=1,\ldots,d$.
\end{observation}


We are now going to introduce the concept of a skeleton of $C$. The infima
\begin{equation}\label{inf}
\rho(c_i) := \inf \, \{ |x -c_{i}|: \ x \in C_{i} \} \in |K|, \ \ i=1,\ldots,d,
\end{equation}
do not exist in general. Instead we shall consider the Dedekind cuts $\rho(c_i)$ determined by the closed downward subsets of all elements of $|K|$ which are smaller than the sets
$$ \left \{ |x -c_{i}|: \ x \in \bigcup \, \left\{ \sigma^{-1}(\xi): \ c_{\xi} = c_{i} \right\} \right\}. $$
They only serve to group the centers of the cells under study. For any Dedekind cut $\rho$ on the value group $|K|$ and an $r \in |K|$, the inequalities $r < \rho$ and $r > \rho$ have a clear meaning.

\vspace{1ex}

Write the 0-definable set of the values $\rho(c_i)$ in the ascending order:
$$ r_0 := 0 < r_1 < r_2 < \dots < r_s; $$
take $r_0 =0$ if $0$ occurs among the $\rho(c_i)$. For $i=0,\ldots,s$, let
$$ c_{i,j} \ \ \text{with} \ \ j=1,\ldots,t_{i}, $$
be the centers for which $\rho(c_{i,j}) =r_{i}$; obviously, $s \leq d$ and $$ t_0 + t_1 + \ldots + t_s = d. $$

\begin{remark}\label{rem-single}
Every finite set
$$ \{ c_{i,1},\ldots,c_{i,t_{i}} \}, \ \ i=0,\ldots,s, $$
is 0-definable, because so is every single Dedekind cut $r_{i}$ in the presence of order relation.
\end{remark}

We can canonically define, by modifying the initial centers, a {\em skeleton} $S(C)$ of $C$, which is a 0-definable set of centers satisfying the following elementary metric condition:
\begin{equation}\label{skeleton}
  |c_{i,j} - c_{p,k}| > r_{i} \ \text{and} \ |c_{i,j} - c_{p,k}| > r_{p},
\end{equation}
which plays a crucial role in Lipschitz extension.

\vspace{1ex}

Skeleton will be built through successive steps taking into account the Dedekind cuts $r_{i}$ considered above. Now we outline its construction.

\vspace{1ex}

$\bullet$ At level~0, all the centers $c_{0,j}$ will be added to the skeleton without modification.

\vspace{1ex}

$\bullet$ At level~1, first remove from the centers $c_{1,j}$, for construction of skeleton only, all those for which
$$ |c_{1,j} - c_{0,k}| \leq r_1 $$
for some $k$. Note that, for any such removed center $c_{1,j}$, the cell $C_{1,j}$ is also a cell with center $c_{0,k}$ (with the same set $R_{1,j}$).
Next notice that the equivalence relation
$$ a \sim b \ \Leftrightarrow \ |a-b| \leq r_{1} $$
partitions the set of the remaining centers from among $c_{1,j}$ into equivalence classes. Any equivalence class is a maximal subset of that set, say $c_{1,j_1}, \ldots, c_{1,j_l}$, such that
$$ |c_{1,j_k} - c_{1,j_q}| \leq r_1 \ \ \text{for all} \ \ 1 \leq k, \ q \leq l. $$
It is easy to check that either the arithmetic average
$$ c := \frac{c_{1,j_1} + \ldots + c_{1,j_l}}{l} $$
lies in no cell $C_{1,j_1}, \ldots, C_{1,j_l}$, or it lies in a ball of radius $r_{1}$ from exactly one of those cells, say $C_{1,j_{k}}$. We can thus replace each center $c_{1,j_q}$ from such a maximal set with $c$ or $c_{1,j_{k}}$ (without change of the set $R_{1,j_q}$), according as the former or latter condition holds, and add such a new center to the skeleton. In this manner, any such equivalence class determines one point to be added to the skeleton.
By Remark~\ref{rem-single}, the set of centers from the skeleton, obtained at this level~1, form a 0-definable set satisfying condition~\ref{skeleton} for $0 \leq i,p \leq 1$.

\vspace{1ex}

$\bullet$ At level 2, first remove from the centers $c_{2,j}$, for construction of skeleton only, all those for which
$$ |c_{2,j} - c_{0,k}| \leq r_2 \ \ \text{or} \ \ |c_{2,j} - c_{1,k}| \leq r_2  $$
for some $k$. And then reason as before.

\vspace{1ex}

$\bullet$ Further, repeat the above construction at level 3 for the centers $c_{3,j}$, and so on.

\vspace{1ex}

This canonical process leads eventually to a unique 0-definable set of centers
$$ S(C) = \{ \tilde{c}_{1},\ldots, \tilde{c}_{e} \} \ \ \text{with} \ e \leq d, $$
satisfying condition~\ref{skeleton}.
Then we say that the reparametrized cell $C$ is with the {\em skeleton} $S(C)$. Obviously,
the center exchange in this construction does not change the distance from the center to its balls.

\vspace{1ex}

Obviously, the points of the skeleton are not 0-definable individually, but the skeleton as a whole is a 0-definable finite set. And the initial reparametrized open cell $C$ is the finite union of open cells
$$ \{ x_{1} \in K: \ rv(x_{1} - \tilde{c}_{i}) \in R_{i} \}, \ \ i=1,\ldots,e, $$
with a 0-definable family $R_{1} \times \ldots \times R_{e} \subset G(K)^{e}$, called an open cell package determined by the skeleton $S(C)$ and that family. We then say that this 0-definable package is determined by the family
$$  (\tilde{c}_{1},R_{1}),\ldots,(\tilde{c}_{e},R_{e}). $$

Under the above assumptions, the skeleton $S(C)$ of a given re\-pa\-rametrized open cell $C$ depends canonically on the initial set of $d$ pairs
$$ \{ (c_{1},\rho(c_{1})), \ldots, (c_{d},\rho(c_{d})) \}. $$
There are finitely many configurations of metric character, say
$$ \mathfrak{C}_1,\ldots,\mathfrak{C}_N, \ \ N = N(d), $$
of the initial $d$-tuples of pairs $(c_{i},\rho(c_{i}))_{i=1}^{d}$, such that for each of which the skeleton is given by a unique formula (built from the $d$-tuples and some arithmetic means occurring in the construction), which depends only on a configuration. A given reparametrized open cell $C$ can be treated as one whose set of centers is just its skeleton.

\begin{remark}\label{configuration}
In the case of higher dimensions, we shall analyze the above construction uniformly in definable families of reparametrized open cells, with valued field parameters running the affine spaces of lower dimension. This will be done by partitioning the parameter space into 0-definable cells to provide a common configuration, and even to provide the same metric $vK$-parameters describing the configuration on each twisted box of the cells thus resulted.
\end{remark}

In order to achieve the goal set out in the above remark, we need to strengthen the fact that the imaginary sort $RV$ is stably embedded, as below. This follows directly from the proof of~\cite[Proposition~2.6.12]{C-H-R}.

\begin{lemma}\label{rem-stable}
Given a 0-definable subset $X$ of $K^{n} \times RV(K)^{m}$, there exists
a 0-definable map
$$ \chi: K^{n} \to RV(K)^{l}, $$
with some $l \in \mathbb{N}$, such that the $\{ a \}$-definable fiber $X_{a} \subset RV(K)^n$ is $\chi (a)$-definable for any $a \in K^{n}$. Then we obtain the 0-definable family
$$ \{ \chi(a) \} \times X_{a} \subset RV(K)^{l} \times RV(K)^{m}, \ \ a \in K^{n}. $$
\end{lemma}


Indeed, dealing with a 0-definable family of reparametrized open cells $C_{t} \subset K$ with centers $c_{1}(t),\ldots,c_{d}(t)$ and with a valued field parameter $t \in T$, we get a 0-definable family
$$ (\chi(c_{i}(t)),\rho(c_{i}(t))), \ \ i=1,\ldots,d, \ t \in T, $$
of pairs with Dedekind cuts, which are imaginary elements of the auxiliary sort $RV$. Hence and by Theorem~\ref{cell-decomposition}, we obtain the following


\begin{observation}\label{config-code}
Under these assumptions, there exists a finite decomposition of $T$ into 0-definable reparametrized cells such that the centers $c_{1}(t),\ldots,c_{d}(t)$ form a common configuration (even with the same metric $vK$-parameters describing the configuration) on each twisted box of those cells. This will be used later in the construction of a skeleton for higher dimensions.
\end{observation}


By the LE-property (corollary~\ref{LEP}), the extension problem on the affine line comes down to extending Lipschitz 0-definable functions from such a 0-definable open cell package.
We shall still need a basic, though elementary, proposition whose proof uses the concept of a skeleton.

\begin{proposition}\label{1-cell}
Let $\mathcal{B}$ be a 0-definable family of pairwise disjoint open balls in $K$ which satisfies the following condition:

For all balls $B_1,B_2 \in \mathcal{B}$, $B_1 \neq B_2$, with radii $r_1, r_2 \in |K| \setminus \{ 0 \}$, respectively, we have
$$ |b_1 - b_2| = \max \{ r_1, r_2 \} \ \ \text{for every} \ \  b_1 \in B_1, \ b_2 \in B_2. $$
Then $\mathcal{B}$ is a 0-definable cell in $K$.
Since Hensel minimality admits adding constants, the conclusion holds for a $\xi$-definable family as well.
\end{proposition}

\begin{proof}
The union of $\mathcal{B}$ is a finite union of 0-definable reparametrized open cells $C_{i}$. It follows from the assumed condition that the skeleton of each of those reparametrized cells consist of one point, and that all those points must coincide. Hence the conclusion follows.
\end{proof}

As observed in Remark~\ref{rem-risometry}, the image $f(B)$ of an open ball $B$ under a definable risometry $f:B \to K$ is an open ball of the same radius. Hence and by Proposition~\ref{1-cell}, we immediately obtain the following corollaries, which will be used (along with model theoretic compactness) in higher dimensions as well.

\begin{corollary}\label{r1}
Let $f:C \to K$ be a 0-definable risometry of a 0-definable open cell $C \subset K$, with a center $c$, onto its image. Then $f(C)$ is a 0-definable cell with a center $d$, and the function
$$ \tilde{f}(x) = \left\{ \begin{array}{cl}
                        f(x) & \mbox{ if } \ x \in C, \\
                        d & \mbox{ if } \ x=c,
                        \end{array}
               \right.
$$
is a risometry onto its image too. The same conclusion holds for a $\xi$-definable open cell as well.   \hspace*{\fill} $\Box$
\end{corollary}



\begin{corollary}\label{r2}
Let $f:C \to K$ be a 0-definable risometry of a 0-definable reparametrized open cell $C$, with a skeleton $S(C)$, onto its image. (By Observation~\ref{cell-fin}, $C$ is actually a finite disjoint union of open cells $C_{i}$, $i=1,\ldots,d$.) Then the image $D:= f(C)$ is a 0-definable reparametrized open cell which is the finite disjoint union of open cells
$$ D_{i} := f(C_{i}), \ \ i=1,\ldots,d. $$
Further $f$ induces a 0-definable function 
$$ f_{s}: S(C) \to S(D) $$
between the skeletons of $C$ and $D$, respectively, and the function
$$ \tilde{f}(x) = \left\{ \begin{array}{cl}
                        f(x) & \mbox{ if } \ x \in C, \\
                        f_{s}(x) & \mbox{ if } \ x \in S(C),
                        \end{array}
               \right.
$$
is a risometry of $C \cup S(C)$ onto $D \cup S(D)$ too.
\end{corollary}

\begin{proof}
While the first conclusion follows directly from Corollary~\ref{r1}, the second one is a straightforward consequence of the very construction of a skeleton.
\end{proof}

\vspace{1ex}

Now we can readily extend a given risometry $f: C \to K$ whose domain $C$ is a 0-definable reparametrized open cell $C \subset K$. Consider its extension
$$ \tilde{f}: C \cup S(C) \to D \cup S(D) $$
from Corollary~\ref{r2}. Then the function $F:K \to K$ given by the formula
$$ F(x) = \left\{ \begin{array}{cl}
                        \tilde{f}(x) & \mbox{ if } \ x \in C \cup S(C), \\
                        \frac{\sum \tilde{f}(S(x))}{\# \, S(x)} & \mbox{ otherwise, }
                        \end{array}
               \right.
$$
where
$$ S(x) := \{ a \in S(C): \ |x-a| = \min \, \{ |x-s|: \ s \in S(C) \} \}, $$
is a 0-definable 1-Lipschitz extension of $f$ we are looking for. This completes the extension problem on the affine line.
  \hspace*{\fill} $\Box$

\vspace{2ex}

The above extension procedure can be split into the following two steps, which will be especially convenient for higher dimensional ambient spaces. First observe that the function
$$ g: K \ni x \mapsto \frac{\sum \tilde{f}(S(x))}{\# \, S(x)} \in K $$
is 1-Lipschitz continuous, and next that so is the extension of
$$ \tilde{f}-g|(C \cup S(C)) $$
by zero through the complement $K \setminus (C \cup S(C))$.



\vspace{2ex}


\section{Lipschitz extension on the affine plane}

Consider a 0-definable 1-Lipschitz function $f:A \to K$ with $A \subset K^{2}$.
By the LE-property, we can assume that $A =C$ is a 0-definable re\-para\-metrized Lipschitz cell of dimension $k=0$, $k=1$ or $k=2$.

\vspace{2ex}

{\bf Case~I.} When $k=\dim C =0$, the set $A$ has a finite number $d$ of points, say
$$ A = \{ (u_{i},v_{i,j}) \in K^{2}: \ i=1,\ldots,s, \ j=1,\ldots,d_{i} \}. $$
Even this case is a bit laborious. Denote by
$$ A_{u_{i}} := \{ v_{i,1},\ldots, v_{i,d_{i}} \} $$
the fiber of $A$ over $u_i$. Put
$$ \delta_{1} := \min \{ |u_i - u_j|: \ i \neq j \}, \ \ \delta_{2} := \min \{ |u_i - u_j|: \ |u_i - u_j| > \delta_{1} \}, $$
and so on. In this fashion, we get a finite increasing sequence
$$ \delta_{1} < \delta_{2} < \ldots < \delta_{t} < \delta_{t+1}= \infty, \ \ t \leq s. $$

We shall canonically extend the function $f$ on the successive $\delta_{k}$-neighbourhoods $A^{\delta}$ of $A$, where
$$ A^{\delta} := \bigcup_{x \in A} \, B(x,\delta), \ \ B(x,\delta) := \{ y \in K^{2}: \ |y-x| < \delta \}. $$

Observe that if $\Phi: E \to K$ is a 1-Lipschitz function on a finite subset $E$ of $K$, then so is
its extension
\begin{equation}\label{ext-fin}
\omega(\Phi): K \ni v \mapsto \frac{\sum f(E(v))}{\# \, E(v)} \in K,
\end{equation}
where
$$ E(v) := \{ w \in E: \ |v-w| = \min \, \{ |v-s|: \ s \in E \} \}, $$
is a 1-Lipschitz extension of $\phi$; this is just formula~\ref{ext-fin-0}.

\vspace{1ex}

For 1-Lipschitz functions $\phi_{i}: E_{i} \to K$ on finite sets $E_{i} \subset K$, $i=1,2,\ldots,p$ and $\delta \in |K|$, define the function
$$ \Phi = \Phi(\phi_{1},\delta, \phi_{2},\ldots, \phi_{p}): \, \bigcup_{i=1}^{p} E_{i} \to K $$
by putting
$$ \Phi(v) = \phi_{1}(v) \ \ \text{if} \ \ v \in E_{1}, $$
and
$$ \Phi(v) = \frac{\sum_{i=2}^{p} \sum \phi_{i}(B \cap E_{i})}{\sum_{i=2}^{p} \# \, (B \cap E_{i})} \ \ \text{if} \ \ v \in B \cap \bigcup_{i=2}^{p} E_{i}, $$
for any open ball
$B \subset \bigcup_{i=2}^{p} \, E_{i}^{\delta} \setminus E_{1}^{\delta}$
of radius $\delta$.

\vspace{1ex}

We begin by extending a 0-definable 1-Lipschitz function $f:A \to K$ from $A$ on the $\delta_{1}$-neighbourhood of $A$. Consider any maximal subset $B_{1}$ of $\pi_{<2}(A) = \{ u_{1}, \ldots, u_{s} \}$ that satisfies the following condition
$$ (\Lambda_{1}) \ \ \ |u-u'|= \delta_{1} \ \ \text{for every} \ u,u' \in B, \ u \neq u'; $$
say, $B_{1} = \{ u_1,\ldots,u_p \}$. Consider the functions
$$ \phi_{i}: A_{u_{i}} \to K, \ \ \phi_{i}(v_{i,j}) := f(u_{i},v_{i,j}), \ \ i=1,\ldots,p, \ j=1,\ldots,d_{i} $$
and  
$$ \Phi = \Phi(\phi_{1},\delta_{1},\phi_{2},\ldots,\phi_{p}). $$
Put
$$ \tilde{f}_{B_{1},1}(u,v) = \omega(\Phi)(v), \ \ (u,v) \in \bigcup_{k=1}^{p} \bigcup_{j=1}^{d_{k}} \, B(u_{1},\delta_{1}) \times B(v_{k,j},\delta_{1}); $$
note that the function $\tilde{f}_{B_{1},1}(u,v)$ is in fact constant with respect to the variable $u \in B(u_{1},\delta_{1})$.

\vspace{1ex}

In a similar fashion, we can canonically define the functions
$$ \tilde{f}_{B_{1},i}(u,v) = \omega(\Phi)(v), \ \ (u,v) \in \bigcup_{k=1}^{p} \bigcup_{j=1}^{d_{i}} \, B(u_{i},\delta_{1}) \times B(v_{k,j},\delta_{1}), $$
and next the function $\tilde{f}_{B_{1}}$ by gluing the functions $\tilde{f}_{B_{1},1}, \ldots, \tilde{f}_{B_{1},p}$ just constructed. Eventually, we can canonically define a function $\tilde{f}_{1}$ by gluing the functions $\tilde{f}_{B_{1}}$ where $B_{1}$ runs over all maximal sets that satisfy condition~$\Lambda_{1}$.
It is not difficult to check that all the functions $\tilde{f}_{B_{1}}$ are 1-Lipschitz on the $\lambda_{1}$-neighbourhood of the finite set
$$ A(B_{1}) := B_{1} \times \{ v_{i,j}: \ i=1,\ldots,p, \ j=1,\ldots,d_{i} \}, $$
and $\tilde{f}_{1}$ is a 0-definable 1-Lipschitz extension of $f$ to the $\delta_{1}$-neighbour\-hood of a finite subset of $K^2$ containing $A$.

\vspace{1ex}

Next consider any maximal subset $B_{2}$ of $\pi_{<2}(A) = \{ u_{1}, \ldots, u_{s} \}$ that satisfies the following condition
$$ (\Lambda_{2}) \ \ \ |u-u'| \leq \delta_{2} \ \ \text{for every} \ u,u' \in B, \ u \neq u'. $$
Every such subset $B_{2}$ is a union of subsets $B_{1}$ considered above. Since the sets $A(B_{1})$ have constant fibres over the points of $B_{1}$, we can now handle subsets $B_{1}$ of the sets $B_{2}$ similarly as we did points of the sets $B_{1}$ before. In this fashion, again we can canonically achieve 1-Lipschitz extensions $\tilde{f}_{B_{2}}$ of
$\tilde{f}_{1}$ on the $\lambda_{2}$-neighbourhood of a finite set $A(B_{2})$ with constant fibres over the points of $B_{2}$, and then a unique 0-definable 1-Lipschitz extension $\tilde{f}_{2}$ of $f$ to the $\delta_{2}$-neighbourhood of a finite subset of $K^2$ containing $A$.

\vspace{1ex}

Repeating the above, we eventually achieve a 0-definable 1-Lipschitz extension $\tilde{f}_{t+1}$ of $f$ to the affine space $K^{2}$, we are looking for.
   \hspace*{\fill} $\Box$



\begin{remark}\label{config-ext}
There are finitely many configurations (cf.~Remark~\ref{configuration}) of $d$ points of the graphs of 1-Lipschitz 0-definable functions $f$ on finite sets $A$, based on the distances between them, such that for each of which the foregoing canonical extension procedure yields a 1-Lipschitz 0-definable extension of $f$ given by a unique formula depending on those points. In a similar fashion, this procedure carries over to the canonical Lipschitz extension from finite 0-definable subsets $A$ of the affine spaces $K^{n}$.
\end{remark}


\vspace{2ex}

{\bf Case~II.} Suppose $k=\dim C =1$. Making use of Corollary~\ref{LEP} and Proposition~\ref{graph}, we can assume that $C$ is a disjoint union of the form
$$ C = \bigcup_{\xi \in \Xi} \, \bigcup_{j=1}^{d} \, \mathrm{graph}\, \phi_{\xi,j} \subset K^2, $$
where
$$ D = \bigcup_{\xi \in \Xi} \, D_{\xi} \subset K^{1}_{x_{1}} $$
is a 0-definable reparametrized open cell in $K_{x_{1}}$. We can of course assume that $D$ is a 0-definable open cell package with a skeleton
$$ S(D) = \{ c_1, \dots, c_s \} \subset K_{x_1}. $$
The construction of centers in the definition of a skeleton and a package, introduced in Case III below in this section, can be repeated almost verbatim to the above family of functions. (But here, only the question of 1-Lipschitz continuity matters.)
We can thus assume that $C$ is a 0-definable package which consists of finitely many graphs of functions
$$ \mathrm{graph}\, \phi_{j}(x_{1}), \ \ j=1,\ldots,d, $$
defined on $D$ such that every function $\phi_{j}$ is 1-Lipschitz on each ball of $D$.
The union of images
$$ \bigcup_{j=1}^{d} \, \phi_{j}(D) \subset K_{x_{2}} $$
is the union of a 0-definable open cell package $E$, with a skeleton
$$ S(E) = \{ d_1, \dots, d_t \} \subset K_{x_2}, $$
and of a finite number of points $d_{t+1},\ldots,d_{u} \in K_{x_2}$.

\vspace{1ex}

As explained in Section~1 after Theorem~\ref{ext}, we can require that, after a suitable modification, $f$ satisfy condition~\ref{risometry} with respect to the variable $x_{1}$. In other words, we can require that the function $f(x_{1},\phi_{j}(x_{1}))$ be a risometry for each $j=1,\ldots,d$. The image $f(C)$ is the union of a 0-definable open cell package $H$ with a skeleton $S(H) = \{ a_1, \dots, a_p \}$ and of a finite number of points.


\vspace{1ex}

First, we shall extend $f$ on the set $P$ of relevant pairs $(c_{i},d_{j})$. Clearly,
$$ D_{i} := \{ x \in D: \ |x_{1} - c_{i}| < \min \{ |c_{j} - c_{i}|: \; j \neq i \} \}, \ \ i=1,\ldots,s, $$
is a family of non-empty sets. For any fixed $i$, the union of images
$$ \bigcup_{j=1}^{d} \, \phi_{j}(D_{i}) \subset K_{x_{2}} $$
is the union of an open cell package $E_{i}$, whose skeleton $S(E_{i})$ is a subset of $S(E)$, and a finite number of points from among $d_{t+1},\ldots,d_{u} \in K_{x_2}$.

\vspace{1ex}

For a fixed $d_{j} \in S(E_{i})$, take a ball $B$ from $E_{i}$ that is next only to $d_{j}$ and a ball $B'$ from $D_{i}$ that meets one of the pre-images $\phi_{j}^{-1}(B)$, $j=1,\ldots,d$.
Since the centers under study are skeletons and $f$ is a risometry with respect to $x_{1}$, the image
$$ f(\mathrm{graph}\, \phi_{j}|B') $$
is next to a unique center from the skeleton $S(H)$, say $a_{k}$. It is not difficult to check that $a_{k}$ is independent of the choice of $B'$. We can thus extend $f$ to a 1-Lipschitz function by putting
$$ \tilde{f}(c_{i},d_{j}) := a_{k}, \ \ d_{j} \in S(E_{i}). $$
Similarly, for a relevant point $d_{j}$ with $j \in \{ t+1,\ldots,u \}$, we put
$$ \tilde{f}(c_{i},d_{j}) := a_{k} $$
when the image
$$ f(\mathrm{graph}\, \phi_{j}^{-1}(d_{j})) $$
is next to a unique center $a_{k} \in S(H)$.

\vspace{1ex}

But we have already considered the case where $A$ is a subset of the affine plane of dimension $0$ (Case~I). Therefore there exists a 0-definable 1-Lipschitz function $g:K^2 \to K$ that agrees with the function $f$ on the set $P$ of the relevant pairs $(c_{i},d_{j})$. This reduces the problem to the case where $f$ is a 1-Lipschitz function on $C \cup P$ which vanishes on $P$.
Now let
$$ f_{x_{1}}: C_{x_{1}} \ni x_{2} \mapsto f(x) \in K, \ \ x_{1} \in D, $$
be the restriction of $f$ to the fiber $C_{x_{1}}$ of $C$ over $x_{1}$.
Then the function $F: K^{2} \to K$ given by the formula
$$ F(x) = \left\{ \begin{array}{cl}
                        \omega(f_{x_{1}})(x_{2}) & \mbox{ if } \ x_{1} \in D, \\
                        0 & \mbox{ if } \ x_{1} \not \in D,
                        \end{array}
               \right.
$$
is a 1-Lipschitz extension of $f$ we are looking for; here the function $\omega(\Phi)$ is given by formula~\ref{ext-fin}. Indeed, consider two points $x,y \in K^{2}$.
Suppose $x_{1}$ and $y_{1}$ lie in distinct balls of the package $D$ with centers $c_{i},c_{l} \in S(D)$, respectively; equality $i=l$ is allowed. Then
$$ |x-y| \geq |x_{1} - y_{1}| \geq \max \, \{ |x_{1} - c_{i}|, |y_{1} - c_{l}| \}, $$
and hence the assertion follows immediately. The case where one of the points does not lie in $D$ is similar. Finally, if $x$ and $y$ lie in a common ball of $D$, the assertion follows from that all the functions $\phi_{1},\ldots,\phi_{d}$ are 1-Lipschitz on every ball of $D$. This solves Case~II.    \hspace*{\fill} $\Box$



\vspace{2ex}

{\bf Case~III.} Finally, suppose $k = \dim C =2$, i.e.\ $C$ is a 0-definable reparametrized open cell in $K^2$:
$$ C = \bigcup_{\xi} \, C_{\xi} \subset K^2, \ \  C_{\xi} = \sigma^{-1}(\xi), \ \ \xi \in \sigma(C),  $$
where $\sigma: C \to RV(K)^{s}$ is a 0-definable function. Every set
$$ C_\xi := \sigma^{-1}(\xi), \ \ \xi \in \sigma(C), $$
is a $\xi$-definable Lipschitz open cell with some center tuple $c_{\xi}$ of the form
\begin{equation}\label{cell-open}
  C_{\xi} = \left\{ x \in K^{2}: (rv(x_{1} - c_{\xi,1}), rv(x_{2} - c_{\xi,2}(x_{1}))) \in R_{\xi} \right\}
\end{equation}
with some 0-definable families of center tuples $c_{\xi} = (c_{\xi,1},c_{\xi,2})$ and of sets
$R_{\xi} \subset G(K)^{2}$.

\vspace{1ex}

We begin by defining {\em a skeleton} of the repara\-me\-trized Lipschitz open cell $C$. It will be done via a suitable finite 0-definable partitioning of the set
$$ D := \pi_{1}(C) = \bigcup_{\xi} \, D_{\xi}, \ \ \text{where} \ \ D_{\xi} := \pi_{1}(C_{\xi}) \subset K_{x_{1}}, $$
up to a subset of lower dimension, which will not affect our solution to the extension problem in view of Corollary~\ref{LEP} and the induction hypothesis. The restriction
$$ p: \bigcup_{\xi} \, \mathrm{graph} (c_{\xi,2}) \to D $$
of the projection $\pi_{1}$  has of course finite, thus uniformly bounded fibres; say, of maximum cardinality $l$.
Since $D$ is the union of the interiors $\mathrm{int} (D_{d})$ of the sets
$$ D_{d} := \{ x_{1} \in D: \; \# \, p^{-1}(x_{1})= d \}, \ \ d \leq l, $$
and of a (finite) set of dimension $< 1$, we can assume that $D = \mathrm{int} (D_{d})$, and thus, as before in Case~II, that the constant cardinality of the fibers $p^{-1}(x_{1}) = C_{x_{1}}$ is $d$ for all $x_{1} \in D$. Then the 0-definable family
$$ D_{\bar{\xi}} := \left\{ x_{1} \in \bigcap_{j=1}^{d} \, D_{\xi_{j}}: \ \# \, \{ c_{\xi_1,2}(x_{1}),\ldots,c_{\xi_d,2}(x_{1}) \} = d \right\}, $$
where $\bar{\xi} = (\xi_1,\ldots,\xi_d) \in \sigma(C)^{d}$, covers $D$. Indeed, we have
$$ D_{\lambda} = \bigcup_{\lambda \in \bar{\xi}} \, D_{\bar{\xi}} \ \ \text{for every} \ \ \lambda \in \sigma(C). $$


\vspace{1ex}

Next, we wish to find a certain special 0-definable cell decomposition of $D$, which is finer than the covering $\{ D_{\bar{\xi}}: \ {\bar{\xi}} \in \sigma(C)^{d} \}$. To this end, consider the following two 0-definable families $\Lambda$ and $\Theta$ with parameter $a \in D$ and fibres in the auxiliary sort $RV$:
$$ \Lambda_{a} := \{ \bar{\xi} \in \sigma(C)^{d}: \, a \in D_{\bar{\xi}} \} \subset \sigma(C)^{d}  $$
and
$$ \Theta_{a} := \{ (\bar{\xi}, rv(a-c_{\xi_{1}}), \dots, rv(a-c_{\xi_{d}})): \, \bar{\xi} = (\xi_{1},\ldots,\xi_{d}) \in \Lambda_{a} \} $$
$$  \subset \sigma(C)^{d} \times G(K)^{d}. $$
Next define the following 0-definable equivalence relation on $D$:
$$ R(a,b) \ \ \Longleftrightarrow \ \ \Theta_{a} = \Theta_{b}, \ \ a,b \in D, $$
which induces the 0-definable partition of $D$ into its equivalence classes
$$ T_{a} = \{ b \in D: \ \Theta_{a} = \Theta_{b} \}, \ \ a \in D. $$
Now, as explained in Remark~\ref{configuration} and Observation~\ref{config-code}, it follows from Theorem~\ref{cell-decomposition} and Lemma~\ref{rem-stable} that there exists a finite decomposition of $D$ into 0-definable reparametrized Lipschitz cells, which is compatible both with the above partition and with the configurations $\mathfrak{C}(C_{a},\bar{\xi}(a))$, i.e.\ both the fibers $\Theta_{a}$ and configurations $\mathfrak{C}(C_{a},\bar{\xi}(a))$ with respect to the construction of a skeleton on the affine line, $\bar{\xi}(a) \in \Lambda_{a}$, are constant over each twisted box $B$ (here, ball in the affine line) of those cells.

\vspace{1ex}

By the LE-property and induction hypothesis, throwing away pieces of lower dimension, we can assume that $D$ is one of those reparametrized open cell. It follows immediately from the very construction that each such ball $B$ lies in a unique equivalence class $T_{a}$, and that the centers $c_{\xi_{i},1}(x_{1})$ are 1-Lipschitz on $B$ for every $\bar{\xi} = (\xi_{1},\ldots,\xi_{d}) \in \Lambda_{a}$.
Hence and by the constancy of fiber configuration, the punctual construction of a skeleton from Section~4, performed here over parameters $x_{1} \in B$, leads to $e$ new centers
$$ \tilde{c}_{1,2}(x_{1}),\ldots, \tilde{c}_{e,2}(x_{1}) $$
with a number $e \leq d$ independent of $x_{1}$, which are produced from the initial $d$ centers by the same formulae and which are 1-Lipschitz too. Indeed, the new centers of the skeleton are built as arithmetic means of a number of the initial ones.

\vspace{1ex}

We call the graphs of those new centers (or their union) {\em a skeleton} of $C$ at level 2 (i.e.\ with respect to the variable $x_2$). Similarly as before on the affine line, we can actually assume that the reparametrized cell $C$ has fibers just of constant cardinality $e$ over $D$. For simplicity, we drop the tilde sign.

\vspace{1ex}

Further, we can also assume that $D$ itself is a 0-definable package determined by a family
$$ (c_{1,1},R_{1}),\ldots,(c_{q,1},R_{q}). $$
Summing up, the skeleton consists of the same number $e$ of graphs of new centers that are 1-Lipschitz over each twisted box (here, ball in the affine line) of the package $D$.
We thus get the following presentation
$$   C = \bigcup_{i=1}^{q} \bigcup_{j=1}^{e} \, C_{i,j}, \ \ \ C_{i,j} := \bigcup_{\lambda_{1} \in R_{i}} \, C_{i,j,\lambda_{1}}, $$
for the 0-definable family
$$ C_{i,j,\lambda_{1}} := \{ x \in K^{2}: \ rv(x_{1} - c_{i,1}) = \lambda_{1}, \ rv(x_{2} - c_{j,2}(x_{1}))) \in R_{i,j,\lambda_{1}} \}, $$
where $R_{i,j,\lambda_{1}}$ is defined by the formula
$$ \{ \lambda_{2} \in G(K): \ \ \exists \; a \in D \ \exists \; \xi \in \bar{\xi} \in \Lambda_{a} \ \ [ rv(a-c_{i,1}) = \lambda_{1}, \ (\lambda_{1},\lambda_{2}) \in R_{\xi} ] \}; $$
independence from the point $a$ follows directly from that the package $D$ is compatible with the partition $\Theta$.

\vspace{1ex}

We shall then say that the reparametrized cell $C$ is a 0-definable open cell package determined by the finitary skeleton
$$ (c_{i,1}), \ (c_{j,2}(x_{1})), \ \ i=1,\ldots,q, \ j=1,\ldots,e, $$
and the infinitary part
$$ (R_{i},  R_{i,j,\lambda_{1}}), \ \ i=1,\ldots,q, \ j=1,\ldots,e, \ \lambda_{1} \in R_{i}; $$
the latter can be expressed equivalently by the 0-definable family
$$ R_{i,j} := \{ \lambda =(\lambda_{1},\lambda_{2}) \in G(K)^{2}: \ \lambda_{2} \in R_{i,j,\lambda_{1}} \} \subset G(K)^{2} $$
with $i=1,\ldots,q$, $j=1,\ldots,e$. Hence we get
$$ C_{i,j} = \{ x \in K^{2}: \ (rv(x_{1} - c_{i}),rv(x_{2} - c_{j}(x_{1}))) \in R_{i,j} \}. $$
Furthermore, each center $c_{j,2}$ is 1-Lipschitz on every twisted box (here ball) of the package $D$.

\vspace{1ex}

By the LE-property and induction hypothesis, the extension problem in Case~III comes down, via throwing away pieces of lower dimension, to extending Lipschitz 0-definable functions from such a 0-definable open cell package.  

\begin{remark}\label{inter-rem}
The above constructions of a total skeleton and a 0-definable open cell packages (with relevant twisted boxes, and not only with balls as in the planar case), can be recursively repeated for 0-definable repara\-metrized open cells in the affine spaces $K^k$. In the next section, this will be applied to the extension problem on the affine spaces $K^k$.
\end{remark}

As indicated before, we can assume without loss of generality that the function $f$ satisfies condition~\ref{risometry} with respect to the variable $x_{2}$.  
Hence and by Corollary~\ref{r2}, the images
\begin{equation}\label{center-family}
 f(C_{i,j} \cap (\{ x_{1} \} \times K)), \ \ x_{1} \in D_{i},
\end{equation}
form a family of cells. Since the defining formulae for their centers depend on configurations rather than on parameters (canonical character of the construction of a skeleton), this is,
by model theoretical compactness, a 0-definable family of cells with some 0-definable family of centers
$$ d_{i,j}(x_{1}), \ \ i=1,\ldots,q, \ j=1,\ldots,e. $$
Set
$$ \eth C := \bigcup_{i,j} \, \mathrm{graph}\, (c_{i,j}), $$
and extend $f$ to the function $\tilde{f} : C \cup \eth C \to K$ by putting
$$ \tilde{f}(x) = \left\{ \begin{array}{cl}
                        f(x) & \mbox{ if } \ x \in C, \\
                        d_{i,j}(x_{1}) & \mbox{ if } \ x \in \mathrm{graph}\, (c_{i,j}) \subset \eth C, \\
                        \end{array}
               \right.
$$
At this stage, our objective is to show that $\tilde{f}$ is a 1-Lipschitz extension of $f$.

\vspace{1ex}

To this end, consider the following package property
\begin{equation}\label{package-part-0}
R_{i,j,\lambda_{1}} \subset \{ \lambda_{2} \in G(K): \ |\lambda_{2}| \leq |\lambda_{1}| \} \ \ \text{for all} \ i,j, \lambda_{1}.
\end{equation}


\vspace{1ex}

In order to ensure this package property, we can, without loss of generality due to the LE-property, partition the package $C$ into two 0-definable open cell packages $C'$ and $C''$ via modification of the sets $R_{i,j,\lambda_{1}} \subset G(K)$ by putting
$$ R_{i,j,\lambda_{1}}' := \{ \lambda_{2} \in R_{i,j,\lambda_{1}}: \ |\lambda_{2}| \leq |\lambda_{1}| \} $$
and
$$ R_{i,j,\lambda_{1}}'' := \{ \lambda_{2} \in R_{i,j}: \ |\lambda_{2}| > |\lambda_{1}| \}. $$
The package $C'$ satisfies of course property~\ref{package-part-0}.

\vspace{1ex}

On the other hand, observe that the package $C''$ has constant fibres over each twisted box (here ball) of the package $D''$. To solve this case, we use the induction hypothesis (uniform version) on the dimension $n$ of the ambient space, applied to the centers of the package $C$. Without loss of generality, we can thus assume that the restrictions $c_{j,2}|D_{\lambda_{1}}$ of the centers of the package, are global 1-Lipschitz functions on $K$.
Obviously, they form a 0-definable family of functions (denoted by the same letters for simplicity):
$$ c_{j,2,\lambda_{1}}: K \to K, \ \ i= 1,\ldots, q, \ j =1, \ldots, e, \  \lambda_{1} \in R_{i}. $$
Hence we obtain a finite 0-definable family of the origins
$$ \mathfrak{O} := \{ (c_{i,1},c_{j,2,\lambda_{1}}(c_{i,1})): \ i,j,\lambda_{1} \} = \{ O_{1},\ldots, O_{s} \} \subset K^{2}. $$
It is not difficult to check, after swapping the roles of variables $x_{1}$ and $x_{2}$, that $C''$ is a 0-definable open cell package with 0-definable family of constant centers of the form:
$$ c''_{1} \in K_{x_{2}}, \ \  c''_{2}(x_{2}) = c''_{2} = c_{i,1} \ \ \text{with some} \ \  (c''_{2},c''_{1}) \in \mathfrak{O}. $$
With the variables $x_{1},x_{2}$ swapped in this way, the package $C''$ satisfies property~\ref{package-part-0} too.


\begin{remark}\label{reorder}
Such an exchange of variables, or rather reordering the variables $x_{1},\ldots,x_{n}$, will be applied in the higher dimensional case in order to ensure the analogue of package property~\ref{package-part-0} for dimension $n$.
\end{remark}

Observe now that the very definition of an open cell package with property~\ref{package-part-0}, along with that the centers of the package $C$ are 1-Lipschitz on every relevant twisted box (here ball) of the package, imply the following estimates:
\begin{equation}\label{center-lip-0}
|d_{i,j}(y) - d_{i',j'}(z)| \leq |y-z|, \ \ y \in D_{i}, \ z \in D_{i'},
\end{equation}
if $(i,j) \neq (i',j')$; or if $(i,j) = (i',j')$ and $y,z$ that lie in distinct balls of the package or $y,z$ lie in a common ball
$$ \{ x_{1} \in K: \ rv(x_{1} - c_{i}) = \lambda_{1} \}, \ \ \lambda_{1} \in R_{i}, $$
but $|y-z| \geq |\lambda_{2}|$ for some $\lambda_{2}$ such that $(\lambda_{1},\lambda_{2}) \in R_{i,j}$.

\vspace{1ex}

Hence the function $\tilde{f}$ will be 1-Lipschitz, if so are the functions $d_{i,j}$. And, furthermore, we can analyse each function $d_{i,j}$ separately on each relevant ball $B$ of the package; say, on the ball
$$ B := \{ x_{1} \in K: \ rv(x_{1} - c_{i,1}) = \lambda_{1} \} \ \ \text{for some} \ \lambda_{1} \in R_{1}. $$
We shall use the valuative Jacobian property (cf.~\cite[Lemma~2.8.5]{C-H-R}) applied to global functions extending the functions  $d_{i,j}$ by zero outside their natural domains. Hence there exists a finite 0-definable set $Z \subset K$, independent of a single ball $B$, such that the quotients
\begin{equation}\label{estimate}
   \frac{|d_{i,j}(y) - d_{i,j}(z)|}{|y_1-z_1|}, \ \ i=1,\ldots,q, \ j=1,\ldots,e,
\end{equation}
are constant for any two distinct points $y,z \in B$ that lie in a ball next to $Z$.

\vspace{1ex}

Further, since the centers $c_{j,2}(x_{1})$ are 1-Lipschitz on $B$ and $f$ is a risometry with respect to the variable $x_{2}$, it is easy to check that
\begin{equation}\label{center-lip}
|d_{i,j}(y) - d_{i,j}(z)| \leq |y-z|
\end{equation}
for any two points $y,z \in B$ such that $|y-z| \geq |\lambda_{2}|$ for some $\lambda_{2} \in G(K)$ with $(\lambda_{1},\lambda_{2}) \in R_{i,j}$.

\vspace{1ex}



Hence and by estimates~\ref{estimate}, the function $d_{i,j}$ is 1-Lipschitz on $B$ if there is some $\lambda_{2} \in G(K)$ such that $(\lambda_{1},\lambda_{2}) \in R_{i,j}$ and $|\lambda_{1}| > |\lambda_{2}|$.

\vspace{1ex}

In the other case, we get $|\lambda_{1}| = |\lambda_{2}|$ for all $\lambda_{2} \in G(K)$ such that $(\lambda_{1},\lambda_{2}) \in R_{i,j}$ by package property~\ref{package-part-0}.
Then the family of the images defining the function $d_{i,j}(x_{1})$ is constant, i.e.\ independent of $x_{1}$, because the center $c_{j,2}(x_{1})$ is 1-Lipschitz on $B$; and also constant is the function $d_{i,j}$ because of the canonical character of its construction and estimates~\ref{center-lip}.



\vspace{1ex}

Therefore, in view of estimates~\ref{center-lip-0}, the function $\tilde{f} : C \cup \eth C \to K$
is a 0-definable 1-Lipschitz extension of $f$, as desired. We drop tilde for simplicity.

\vspace{1ex}

Now, since we have already considered the case where the set $A$ is of dimension $<2$, there exists a 0-definable $\epsilon$-Lipschitz function
$$ g:K^2 \to K $$
that agrees with the function $\tilde{f}$ on the set $\eth C$.
In this manner, the problem can be reduced to the case where $f$ is an $\epsilon$-Lipschitz function which vanishes on the set $\eth C$.
Then, finally, since $C$ is an open cell package, the function
$$ F:K^{2} \to K $$
given by the formula
$$ F(x) = \left\{ \begin{array}{cl}
                        f(x) & \mbox{ if } \ x \in C \cup \eth C, \\
                        0 & \mbox{ otherwise, } \
                        \end{array}
               \right.
$$
is a 0-definable $\epsilon$-Lipschitz extension of the function $f$ we are looking for.    \hspace*{\fill} $\Box$

\vspace{1ex}

The uniform version of the theorem follows by model theoretical compactness, because the foregoing constructions are canonical. This completes the proof of Theorem~\ref{ext-uni} for the affine plane.

\vspace{2ex}

\section{Lipschitz extension for higher dimensions}

To establish the extension theorem for higher dimension, we shall proceed with double induction on the dimension $n$ of the ambient space and the dimension $k$ of the subset $A$.
We briefly outline our strategy, which is similar, although more technical compared to the planar case. The canonical extension in the case $k=0$ can be constructed similarly as in the case of the affine plane (cf.~Remark~\ref{config-ext}).

\vspace{1ex}

For the case $k >0$, we shall consider 0-definable open cell packages $C \subset K^{k}$ of dimension $k$ (cf.~Remark~\ref{inter-rem}). Such a package $C$ is determined by a skeleton which consists of the following sequences of centers
$$ (c_{j,1}), \ j=1, \ldots, d_{1}, \ \ldots, \ (c_{j,k}(x_{1},\ldots,x_{k-1})), \ j =1, \ldots, d_{k}, $$
and an infinitary part
$$ R_{j_{1},\ldots,j_{k}} \subset G(K)^{k}, \ \ j_{i} =1, \ldots, d_{i}, \ i=1,\ldots,k, $$
which yield the presentation
$$ C = \bigcup_{j_{1},\ldots,j_{k}} \, C_{j_{1},\ldots,j_{k}} \ \ \text{with} \ \ C_{j_{1},\ldots,j_{k}} := $$
$$ \{ x \in K^{k}: \ (rv(x_{1} - c_{j,1}),\ldots, rv(x_{k} - c_{j,k}(x_{1},\ldots,x_{k-1}))) \in R_{j_{1},\ldots,j_{k}} \}. $$
Moreover, each center is 1-Lipschitz on every relevant twisted box determined by the above presentation.

\vspace{1ex}

Further, we shall need the following package property, being a higher dimensional analogue of property~\ref{package-part-0}:
\begin{equation}\label{package-part-1}
R_{j_{1},\ldots,j_{k}} \subset \{ \lambda \in G(K)^{k}: \ |\lambda_{1}| \geq |\lambda_{2}| \geq \ldots \geq |\lambda_{k}| \}
\end{equation}
for all $j_{1},\ldots,j_{k}$, which can be obtained by reordering the variables, as outlined below.

\vspace{1ex}

The reasoning here is similar, although more technical in comparison with the planar case. We can assume, using the induction hypothesis (uniform version) on the dimension $k$ of the set $A$, that the restrictions of each center
$$ c_{j_{1},\ldots,j_{k}}, \ \ j_{i}=1,\ldots,d_{i}, \  i=2,\ldots,k, $$
of the skeleton to every twisted box of the package contained in $K^{i-1}$ are global 1-Lipschitz functions. This will be illustrated by the following example. Suppose that after an appropriate partition we have obtained the case where
$$ R_{j_{1},j_{2},j_{3},j_{4}} \subset \{ \lambda \in G(K)^{4}: \ |\lambda_{1}| \geq |\lambda_{4}| > |\lambda_{2}| \geq |\lambda_{3}| \}. $$
Then we can replace the center $c_{j_{4},4}(x_{1},x_{2},x_{3})$ by a new one of the form
$$ c_{j_{4},4}(x_{1},c_{j_{2},2}(x_{1}), c_{j_{3},3}(x_{1},c_{j_{2},2}(x_{1}))), $$
which is independent of $x_{2},x_{3}$. With the variables $x_{1},x_{4},x_{2},x_{3}$ reordered in this way, the package is determined by sequences of centers of the form
$$ c_{j_{1},1}, \ c_{j_{4},4}(x_{1},c_{j_{2},2}(x_{1}), c_{j_{3},3}(x_{1},c_{j_{2},2}(x_{1}))), \ c_{j_{2},2}(x_{1}), \ c_{j_{3},3}(x_{1},x_{2}), $$
which satisfies Property~\ref{package-part-1}. This canonical procedure, performed for all sequences of centers, provides eventually a 0-definable package as output.

\vspace{1ex}

Now we can readily return to the proof of the extension theorem. As before, we shall separately treat the cases $0<k<n$ and $k=n$, starting from the latter for pedagogical reasons.

\vspace{1ex}

For $0<k<n$, we shall first proceed with induction on the dimension $n$ of the ambient space (applied here to dimensions $\leq k-1 < n-1$) to globally and uniformly extend some centers under study, in order to ensure Property~\ref{package-part-1} (cf.~Remark~\ref{reorder}); and next, with induction on the dimension $k$ of the set $A \subset K^{n}$ to globally extend the restriction of a given function $f$ on a subset of dimension $k-1$.

\vspace{1ex}

For $k=n$, we shall first proceed with induction on the dimension $n$ of the ambient space (applied here to dimensions $< n$) to globally and uniformly extend some centers under study, in order to ensure package property~\ref{package-part-1}. And next, using induction on the dimension $k$ of the set $A \subset K^{n}$, to globally extend the restriction of a given function $f$ on a subset of dimension $n-1$.

\vspace{1ex}

{\bf Case~I.} Suppose $A \subset K^{n}$ and $\dim A = n$. By the LE-property and induction hypothesis, the extension problem comes down, via throwing away pieces of lower dimension, to extending Lipschitz 0-definable functions from a 0-definable open cell package $A=C$ determined by a skeleton
$$ (c_{j,1}), \ j=1, \ldots, d_{1}, \ \ldots, \ (c_{j,n}(x_{1},\ldots,x_{n-1})), \ j =1, \ldots, d_{n}, $$
and an infinitary part
$$ R_{j_{1},\ldots,j_{n}} \subset G(K)^{n}, \ \ j_{i} =1, \ldots, d_{i}, \ i=1,\ldots,n. $$
As in the planar case, our objective is to extend $f$ on the graphs of the centers $c_{1,n},\ldots,c_{j_{n},n}$ using package property~\ref{package-part-1}.
Again, we can assume without loss of generality that $f$ satisfies condition~\ref{risometry} with respect to the variable $x_{n}$. Hence and by Corollary~\ref{r2}, the images
\begin{equation}\label{center-family-1}
 f(C_{j_{1},\ldots,j_{n}} \cap (\{ (x_{1},\ldots,x_{n-1}) \} \times K)),
\end{equation}
$(x_{1},\ldots,x_{n-1}) \in \pi_{<n}(D_{j_{1},\ldots,j_{n}})$, form a 0-definable family of cells with some 0-definable family of centers
$$ d_{j_{1},\ldots,j_{n}}(x_{1},\ldots,x_{n-1}). $$
Set
$$ \eth C := \bigcup_{j_{1},\ldots,j_{n}} \, \mathrm{graph}\, (c_{j_{1},\ldots,j_{n}}), $$
and extend $f$ to the function  $\tilde{f} : C \cup \eth C \to K$ by putting
$$ \tilde{f}(x) = \left\{ \begin{array}{cl}
                        f(x) & \mbox{ if } \ x \in C, \\
                        d_{j_{1},\ldots,j_{n}}(x_{1},\ldots,x_{n-1}) & \mbox{ if } \ x \in \mathrm{graph}\, (c_{j_{1},\ldots,j_{n}}) \subset \eth C. \\
                        \end{array}
               \right.
$$
At this stage, our objective is to show that $\tilde{f}$ is a 1-Lipschitz extension of $f$.

\vspace{1ex}

As in the planar case, the very definition of an open cell package with property~\ref{package-part-1}, along with that the centers of the package $C$ are 1-Lipschitz on every relevant twisted box of the package, imply the following estimates:
\begin{equation}\label{center-lip-1}
|d_{j_{1},\ldots,j_{n}}(y) - d_{j_{1},\ldots,j_{n}}(z)| \leq |y-z|
\end{equation}
for any $y,z \in K^{n-1}$ that lie in distinct twisted boxes in $K^{n-1}$ of the package; or lie in a common twisted box determined by a center $c_{j_{1},\ldots,j_{n-1}}$ and a value
$$ \lambda' \in \pi_{<n}(R_{j_{1},\ldots,j_{n}}), $$
but $|y-z| \geq |\lambda_{n}|$ for some $\lambda_{n}$ such that $(\lambda',\lambda_{n}) \in R_{j_{1},\ldots,j_{n}}$.

\vspace{1ex}

Hence the function $\tilde{f}$ will be 1-Lipschitz, if so are the functions $d_{j_{1},\ldots,j_{n}}$. And further, we can analyse each function $d_{j_{1},\ldots,j_{n}}$ separately on each relevant twisted box of the package. (We argued likewise in the planar case.)

\vspace{1ex}

Since each twisted box is (uniformly) 1-bi-Lipschitz with the associated box (with zero centers), one can limit oneself to only analyse the function $d := d_{j_{1},\ldots,j_{n}}$ defined on a single box
$$ B := \{ (x_{1},\ldots,x_{n-1}) \in K^{n-1}: \ (rv(x_{1}),\ldots,rv(x_{n-1})) = \lambda' \} $$
with $\lambda' = (\lambda_{1},\ldots,\lambda_{n-1}) \in G(K)^{n-1}$. As before (estimates~\ref{estimate}), we shall use the claim below, which is a parametric version of the valuative Jacobian property applied to global functions extending the functions  $d_{j_{1},\ldots,j_{n}}$ by zero outside their natural domains.

\vspace{1ex}

We still need some notation:
$$ x^{(i)} = (x_{1},\ldots,x_{i-1},x_{i+1},\ldots,x_{n-1})  $$
and
$$ d^{(a^{i})}(x_{i}) := d(a_{1},\ldots,a_{i-1},x_{i},a_{i+1},\ldots,a_{n-1}) $$
for $i=1,\dots,n-1$.


\begin{claim}
There exist nowhere-dense 0-definable subsets
$$ Z_{i} \subset K^{n-1}, \ \ i=1,\ldots,n-1, $$
independent of a single twisted box $B$, such that the projections $\pi_{\neq i}|Z_{i}$ have finite fibres, and for each point $a^{(i)} \in \pi_{\neq i}(B)$ the quotients
\begin{equation}\label{estimate-n}
   \frac{|d^{a^{(i)}}(y) - d^{a^{(i)}}(z)|}{|y-z|}
\end{equation}
are constant for any two distinct points $y,z \in K_{x_{i}}$ that lie in a ball in the affine space  $\{ a^{(i)} \} \times K_{x_{i}}$ next to $Z_{i} \cap (\{ a^{(i)} \} \times K_{x_{i}})$.
\end{claim}

This follows immediately from the ordinary Jacobian property by model theoretical compactness.    \hspace*{\fill} $\Box$

\vspace{2ex}

Further, as in the planar case, since the centers $c_{j_{1},\ldots,j_{n}}$ are 1-Lipschitz on $B$ and $f$ is a risometry with respect to the variable $x_{n}$, we easily obtain the estimates
\begin{equation}\label{center-lip-n}
|d_{j_{1},\ldots,j_{n}}(y) - d_{j_{1},\ldots,j_{n}}(z)| \leq |y-z|
\end{equation}
for any two points $y,z \in B$ such that $|y-z| \geq |\lambda_{n}|$ for some $\lambda_{n} \in G(K)$ with $(\lambda',\lambda_{n}) \in R_{j_{1},\ldots,j_{n}}$.


\vspace{1ex}

Hence and by estimates~\ref{estimate-n}, each function $d_{j_{1},\ldots,j_{n}}$ is 1-Lipschitz on $B$ if
$$ |\lambda_{1}| \geq \ldots \geq |\lambda_{n-1}|  > |\lambda_{n}| $$
for some $\lambda_{n} \in G(K)$ such that $(\lambda',\lambda_{n}) \in R_{j_{1},\ldots,j_{n}}$.

\vspace{1ex}

In the other case, it follows from package property~\ref{package-part-1} that
$$ |\lambda_{1}| \geq \ldots \geq  |\lambda_{p}| > |\lambda_{p+1}| = \dots = |\lambda_{n-1}| = |\lambda_{n}| $$
for all $\lambda_{n} \in G(K)$ such that $(\lambda',\lambda_{n}) \in R_{i,j}$ .
Then the family of the images defining the function $d_{j_{1},\ldots,j_{n}}(x_{1},\ldots,x_{n-1})$ is independent of the variables $x_{p+1},\ldots,x_{n-1}$, because the center
$c_{j_{n},n}(x_{1},\ldots,x_{n-1})$ is 1-Lipschitz on $B$; and so is the function $d_{j_{1},\ldots,j_{n}}(x_{1},\ldots,x_{p})$ because of the canonical character of its construction and estimates~\ref{center-lip-n}.



\vspace{1ex}

Therefore, in view of estimates~\ref{center-lip-1}, the function $\tilde{f} : C \cup \eth C \to K$ is a 0-definable 1-Lipschitz extension of $f$, as desired.
We drop tilde for simplicity.


\vspace{1ex}

By the induction hypothesis, the extension theorem holds for the sets $A$ of dimension $<n$. Hence there exists a 0-definable $\epsilon$-Lipschitz function
$$ g:K^n \to K $$
that agrees with the function $f$ on the set $\eth C$.
In this manner, the problem can be reduced to the case where $f$ is an $\epsilon$-Lipschitz function which vanishes on the set $\eth C$.
Then, finally, since $C$ is an open cell package with package property~\ref{package-part-0}, the function
$$ F:K^{n} \to K $$
given by the formula
$$ F(x) = \left\{ \begin{array}{cl}
                        f(x) & \mbox{ if } \ x \in C \cup \eth C, \\
                        0 & \mbox{ otherwise, } \
                        \end{array}
               \right.
$$
is a 0-definable $\epsilon$-Lipschitz extension of the function $f$ we are looking for.    \hspace*{\fill} $\Box$

\vspace{1ex}

\begin{remark}\label{simultaneous}
The foregoing extension procedure is performed on the line $K_{x_{n}}$, with respect to parameters $(x_{1},\ldots,x_{n-1})$ from $K^{n-1}$, and controlled by the valuative Jacobian property with parameter.
\end{remark}



\vspace{2ex}

{\bf Case~II.} Suppose $A \subset K^{n}$ and $\dim A = k$ with $k \in \{1,\ldots,n-1 \}$ and assume that the extension theorem holds both in the affine spaces of dimensions $<n$, and for subsets of $K^n$ of dimension $<k$.

\vspace{1ex}

As in the planar case, we can assume without loss of generality, using Proposition~\ref{graph} and the LE-property, that
$A$ is a 0-definable package which consists of finitely many graphs $\mathrm{graph}\, \phi_{j}(x_{1},\ldots,x_{k})$ of maps
$$ \phi_{j} = (\phi_{j,1}, \ldots, \phi_{j,n-k}): D \to K^{n-k}, \ \ j=1,\ldots,d, $$
defined on a 0-definable open cell package $D$ in the affine space $K^{k}$, with package property~\ref{package-part-0}, such that every map $\phi_{j}$ is 1-Lipschitz on each twisted box of $D$.

\vspace{1ex}

For the package $D$, we keep the notation from the beginning of this section. We are going to repeat the extension procedure for the plane from Section~5, Case~II, arranged now for a version with parameters $(x_{1},\ldots,x_{k-1})$ from $K^{k-1}$. We shall thus extend $f$ on a subset
$$ \mathfrak{N} \subset \eth D \times K^{n-k} $$ with finite fibers over $\eth D$, which gives control over further 1-Lipschitz extension of the function $f$; here
$$ \eth D := \bigcup_{j=1}^{d_{k}} \, \mathrm{graph}\, (c_{j,k}). $$
However, this situation is more difficult because we are now dealing with the graphs of maps with target $K^{n-k}$. We describe the procedure for one fixed parameter
$$ x' =(x_{1},\ldots,x_{k-1}) \in \pi_{<k}(D). $$
Denote by $D_{x'}$ and $C_{x'}$ the packages in the line $K_{x_{k}}$ and in $K_{x_{k}} \times K^{n-k}$ induced by the package $D$, and by $\eth D_{x'}$ the part of $\eth D$ lying over $x'$.

\vspace{1ex}

As explained in Section~1 after Theorem~\ref{ext}, we can require that $f$ satisfy condition~\ref{risometry} with respect to the variable $x_{k}$; more precisely, the function $f(x_{1}, \dots, x_{k},\phi_{j}(x_{1}, \dots, x_{k}))$ be a risometry with respect to the variable $x_{k}$, $j=1,\ldots,d$.
Then the unions of images
$$ \bigcup_{j=1}^{d} \, \phi_{j,l}(D_{x'}) \subset K_{x_{l}}, \ \ l=k+1,\ldots,n, $$
are some 0-definable open cell packages $E_{x',l}$, with skeletons
$$ S(E_{x',l}) = \{ d_{l,1}, \dots, d_{l,t_{l}} \} \subset K_{x_l}, $$  
and of some finite sets $T_{x',l} = \{ d_{l,t_{l}+1},\ldots,d_{l,u_{l}} \} \subset K_{x_l}$.  

\vspace{1ex}

Further, the image $f(C_{x'})$ is the union of a 0-definable open cell package $H_{x'}$ with a skeleton $S(H_{x'}) = \{ a_1, \dots, a_p \}$ and of a finite set.

\vspace{1ex}

For any point $(x', c_{i,k}(x')) \in \eth D$, the family $D_{x',i}$ of those balls of $D_{x'}$ that are closer to $(x', c_{i,k}(x'))$ than to any other points of $(x', c_{j,k}(x')$, $j \neq i$, is well defined by the very definition of a package. Define the packages $E_{x',i,l}$ for $D_{x',i}$ as the packages $E_{x',l}$ for $D_{x'}$, $l=k+1,\ldots,n$. Their skeletons $S(E_{x',i,l})$ are subsets of $S(E_{x',l})$.
Put
$$ N_{x'} := \eth D_{x'} \times \prod_{l=k+1}^{n} \, (S(E_{x',l}) \cup T_{x',l}) \subset K^{n}. $$
An $n$-tuple $(x', c_{i,k}(x'),d) \in N_{x'}$ is relevant if there exist balls $B_{l}$ next to $d_{l}$ from the packages $E_{x',i,l}$, respectively, such that one of the pre-images
$$ \phi_{j}^{-1}(B_{l} \times \ldots \times B_{n}), \ \ j=1,\ldots,d, $$
meets the package $D_{x',i}$; say meets a ball $B'$ from $D_{x',i}$. Denote by $\mathfrak{N}_{x'}$ the set of those relevant points.

\vspace{1ex}

Now we wish to extend $f$ on the set $\mathfrak{N}_{x'}$. Given a point
$$ (x', c_{i,k}(x'),d) \in \mathfrak{N}_{x'}, $$
take balls $B_{l}$, $l=k+1,\ldots,n$, and $B'$ as above. Since the centers under study are skeletons and $f$ is a risometry with respect to $x_{k}$, the image
$$ \tilde{f}(\mathrm{graph}\, \phi_{j}|B') $$
is next to a unique center from the skeleton $S(H_{x'})$, say $a_{q}$. It is not difficult to check that $a_{k}$ is independent of the choice of $B'$. Then we put
$$ f(x', c_{i,k}(x'),d) := a_{q}. $$

In this manner, we obtain a 0-definable extension $\tilde{f}$ of $f$ on the set
$$ \mathfrak{N} := \bigcup \, \{ \mathfrak{N}_{x'}: \ x' \in \pi_{<k}(D) \}. $$
Note that we argued likewise in the planar case (Section~5, Case~II) with the function $f$ defined on the graphs of a finite 0-definable family of functions $\phi_{j}$, but without parameters $x'$.
The function $\tilde{f}$ is 1-Lipschitz because:

$\bullet$ $D$ is an open cell package with package property~\ref{package-part-0},

$\bullet$ $H_{x'}$, with $x' \in \eth D$, are open cell packages,

$\bullet$ the maps $\phi_{j}$ are 1-Lipschitz,

$\bullet$ and $f$ is a risometry with respect to the variable $x_{k}$.

\noindent
These properties make it possible to control the distance between given two points of the domain versus the distance between other points and sets involved in the construction. We leave the details for the reader. For simplicity, we drop the tilde over $f$.

\vspace{1ex}

By the induction hypothesis, the extension theorem holds for the sets $A$ of dimension $<k$. Hence there exists a 0-definable $\epsilon$-Lipschitz function
$$ g:K^n \to K $$
that agrees with the function $f$ on the set $\mathfrak{N}$.

\vspace{1ex}

In this manner, the problem can be reduced, similarly as on the affine plane, to the case where $f$ is an $\epsilon$-Lipschitz function which vanishes on the set $\mathfrak{N}$.
Then, finally, since $D$ is an open cell package with package property~\ref{package-part-0}, the function
$$ F:K^{n} \to K $$
given by the formula
$$ F(x) = \left\{ \begin{array}{cl}
                        f(x) & \mbox{ if } \ x \in A \cup \mathfrak{N}, \\
                        0 & \mbox{ otherwise, } \
                        \end{array}
               \right.
$$
is the desired 0-definable $\epsilon$-Lipschitz extension of the function $f$.    \hspace*{\fill} $\Box$

\vspace{1ex}

As before, the uniform version of the theorem follows by model theoretical compactness, which completes the proof of Theorem~\ref{ext-uni}.

\vspace{1ex}

We conclude with two comments. Firstly, in view of Example~\ref{ex-ret}, the still open problem of the existence of definable Lipschitz retractions on definable closed subsets of an affine space, with an arbitrarily small Lipschitz constant $>1$, is plausible only for rank one valued fields. Under this assumption, the technique of Lipschitz definable packages with skeletons, introduced in this paper, will allow us to establish both Lipschitz extensions with an arbitrarily small magnification of the Lipschitz constant, and the existence of  Lipschitz retractions with Lipschitz constant arbitrarily close to 1. This will be done via simultaneous induction with respect to dimension in our forthcoming article.

\vspace{1ex}

Secondly, our recent article~\cite{Now-Apal} investigates geometry and  topology of Hensel minimal structures with a natural condition imposed on the auxiliary sort $RV$, covering many classical non-Archimedean structures such as, for instance, valued fields with analytic structure. Among the main results achieved there are the existence of the limit, the closedness theorem, non-Archimedean versions of the \L{}ojasiewicz inequality, an embedding theorem for regular definable spaces, and the definable ultranormality and ultraparacompactness of definable Hausdorff LC-spaces.
Some of those results for Henselian fields, without and with analytic structure, were provided in our previous articles~\cite{Now-Sel,Now-Sing,Now-Alant}.
In~\cite{Now-JSL}, we prove the existence of Yomdin--Gromov parametrizations in the general settings of equicharacteristic zero valued fields with analytic structure. Our approach is based on a definable, non-Archimedean version of Bierstone--Milman desingularization algorithm from our paper~\cite{Now-Sym}.
Theorems on extending continuous definable functions and on definable retractions on closed subsets of an affine space are given in our paper~\cite{Now-closed}.


\vspace{2ex}

\begin{small}
Institute of Mathematics

Faculty of Mathematics and Computer Science

Jagiellonian University

ul.~Profesora S.\ \L{}ojasiewicza 6,

30-348 Krak\'{o}w, Poland

{\em E-mail address: nowak@im.uj.edu.pl}
\end{small}

\end{document}